\documentclass[12pt, a4paper]{amsart}
\usepackage{amsfonts,amsmath, amsthm, amssymb, amscd}
\usepackage{latexsym}
\input{amssym.def}
\input{amssym}
\usepackage{xcolor}
\usepackage[dvips]{graphicx}
\usepackage{hyperref}

\newtheorem{theorem}{Theorem}[section]
\newtheorem{lemma}[theorem]{Lemma}
\newtheorem{cor}[theorem]{Corollary}
\newtheorem{prop}[theorem]{Proposition}

\newtheorem{defi}[theorem]{{Definition}}
\newtheorem{satz}[theorem]{{Theorem}}
\newtheorem{bem}[theorem]{{Remark}}
\newtheorem*{qu}{{Question}}

\newcommand {\N}{\mathbb{N}} 
 
\newcommand {\R}{\mathbb{R}} 
 
\newcommand {\C}{\mathbb{C}}

\DeclareMathOperator{\vol}{vol}

\begin{document}
\title{Differential Operators on non-compact harmonic manifolds}

\author{Oliver Brammen}
\date{\today}
\address{Faculty of Mathematics,
Ruhr University Bochum, 44780 Bochum, Germany}
\email{oliver.brammen@rub.de}
\thanks{
Funded by the Deutsche Forschungsgemeinschaft (DFG, German Research Foundation) - Project-ID: 281071066 -TRR 191}
\begin{abstract}
We study the algebra of differential operators on non-compact simply connected harmonic manifolds and provide sufficient conditions for them to have a radial fundamental solution and be surjective on the space of smooth function. Furthermore, we show that the algebra of differential operators that commute with taking averages over geodesic spheres is generated by the Laplacian. As an application of this, we show that the heat-semi group is dense in the radial $L^1$ space of a non-compact simply connected harmonic manifold. In the process, we provide a characterisation of the eigenfunctions of the Laplacian and therefore of the eigenfunctions of all differential operators commuting with taking spherical averages. 
\end{abstract}


\maketitle

\section{Introduction}
A harmonic manifold is a complete Riemannian manifold $(X,g)$  such that at every point there exists a radial solution of the Laplace equation. This is equivalent to the property that 
for all $p\in X$ the volume density function in geodesic coordinates $\sqrt{g_{ij}(p)}=\theta_q(p)$ only depends on the geodesic distance. From this, it is easy to see that Euclidean and non-flat symmetric spaces of rank one are harmonic. Szabo \cite{szabo1990}  showed that in the compact case, all simply connected harmonic manifolds are symmetric spaces of rank one, thereby proving the so-called Lichnerowicz conjecture \cite{Lich} in the compact setting. But for the non-compact case, there are counterexamples, namely in 1992  Damek and Ricci \cite{Damek_1992} provided for dimension 7 and higher a class of homogeneous harmonic spaces that are non-symmetric. These manifolds were subsequently called Damek Ricci spaces.  In 2006 Heber \cite{Heber_2006} showed that all homogeneous non-compact simply connected harmonic spaces are of the type mentioned above. Since these spaces have a rich algebraic structure one can use the group structure to study the algebra of differential operators (see \cite{helgasonsur}, \cite{helgasonfundamental},\cite{Damek1992},\cite{Anker96}). 
The authors in \cite{PS15} and \cite{Biswas2019} showed that general harmonic manifolds provide a nice playground for integral geometric analysis. In the following, we are going to use the tools developed by the authors in  \cite{PS15}   to study the algebra of differential operators. 
In Theorem \ref{thm:diffop} we show a support theorem for self-adjoint translation invariant differential operators having a radial fundamental solution. For this, we exploit the convolution on harmonic manifolds and employ a classical result by Banach giving sufficient conditions for an operator between Frechet spaces to be surjective. Thereby we generalise a result form   \cite{helgasonsur} from symmetric spaces to harmonic manifolds. 
Theorem \ref{thm:fundamental} then shows that a radial fundamental solution exists for differential operators that commute with taking averages over spheres and have constant coefficients on horospheres. This is done by showing that the Abel transform of the operator has constant coefficients and therefore has a fundamental solution by the Malgrange-Ehrenpreis theorem and then transforming this solution back. This is a generalisation for the results on symmetric spaces \cite{helgasonfundamental}.
Finally in 
Theorem \ref{thm:invariantalgebra} we show that the algebra of Differential operators which commute with taking averages over spheres is generated by the Laplacian. For this, we first show that one can express the even derivatives of a radial function at $r=0$ via the Laplacian and then use this fact to obtain an expression as a polynomial in the Laplacian for every Differential operator satisfying the conditions above.  Lastly, we conclude from Theorem    \ref{thm:invariantalgebra} that the heat semigroup when interpreted as integral kernels is dens in the space of radial $L^1$ functions on $X$. This generalises results for Dameke-Ricci spaces \cite{Damek1992} and symmetric spaces \cite{helgason1994geometric}. In addition to that we provide a characterisation of eigenfunctions of the Laplacian via spherical averages (Theorem \ref{thm:cheigf}). Those results generalise results from symmetric resp. non-symmetric Damek Ricci spaces. See for instance \cite[Chapter II Proposition 2.6]{helgason1994geometric} resp.  \cite[Proposition 2.5.2]{Naik2021}.

\section{Prilimenaries}\label{sec:pr}
 The following introduction into a non-compact harmonic manifold is not intended to be complete. For a concise introduction to the topics, we refer the reader to the surveys \cite{kreyssig2010introduction} and \cite{knieper22016} where also the proofs of the statements in this section can be found. 
 Let $(X,g)$ be a $n$-dimensional non-compact simply connected Riemannian manifold without conjugate points. Denote by $C^k(X)$ the space of $k$-times differentiable functions on $X$ and by $C^k_c(X)\subset C^k(X)$ those with compact support. With the usual conventions for smooth and analytic functions. Furthermore for $x\in X$ denote by $C^k(X,x)$ the functions in $C^k_c(X)$ radial around $x$, i.e $f\in C^k(X,x)$, if there exists an even function $u\in C_{\text{even}}^k(\R)$ on $\R$ such that $f=u\circ d(x,\cdot)$ where $d:X\times X\to\R_{\geq 0}$ is the distance induced by $g$. Furthermore, for $p\geq 1$ $L^p(X)$ refers to the $L^p$-space of $X$ with regards to the measure induced by the metric and integration over a manifold is always interpreted as integration with respect to the canonical measure on this manifold unless stated otherwise and $L^p(X,x)$ is the subspace of functions in $L^p(X)$ radial around $x$. 
   For $p\in X$ and $v\in S_pX$ denote by $c_v :\R\to X$ the unique unit speed geodesic with $c(0)=p$ and $\dot{c}(0)=v$. Define $A_v$ to be the Jacobi tensor along $c_v$ with initial conditions $A_v(0)=0$ and $A^{\prime}(0)=\operatorname{id}$. For details on Jacobi tensors see \cite{KNIEPER2002453}.

\begin{defi}\label{defi:har}
Let $(X,g)$ be a complete non-compact simply connected manifold without conjugate points and $SX$ its unit tangent bundle. For $v\in SX$ let $A_v(t)$ be the Jacobi tensor with initial conditions $A_v(0)=0$ and $A^{\prime}_v(0)=\operatorname{id}$. Then $X$ is said to be harmonic if and only if
$$A(r)=\operatorname{det}(A_v(r))\quad\forall v\in SX.$$
Hence the volume growth of a geodesic ball centred at $\pi(v)$ only depends on its radius.
\end{defi}
\begin{bem}
The volume of the sphere of radius $r$ around $p$ is given by:
\begin{align}\label{eq:volS}
\operatorname{vol}S(p,r)=\int_{S_p X}\operatorname{det} A_v(r)\,dv.
\end{align}
where in this case we assume $dv$ to be the  normalised volume element of $S_p M$.
The second fundamental form of $S(p,r)$ is given by $A_v^{\prime}(r)A^{-1}_v(r)$ and the mean curvature by 
\begin{align}\label{eq:meancurv}
\nu_p(r,v)=\operatorname{trace}A_v^{\prime}(r)A_v^{-1}(r).
\end{align}
From  (\ref{eq:meancurv}) one easily concludes that Definition \ref{defi:har} is equivalent to the mean curvature of geodesic spheres only depending on the radius and is then given by 
$\frac{A'(r)}{A(r)}$. 
\end{bem}

 For $f\in C^2(X)$ the Laplace-Beltrami operator is defined by
\begin{align*}
\Delta f:=\operatorname{div}\operatorname{grad} f
\end{align*}

 $\Delta$ is by definition linear on $C^{\infty}_c(X)$ and we have 
\begin{align}\label{eq:deltagrad}
\int_X -\Delta f(x)\cdot f(x) \,dx=\int_X\lVert \nabla f(x)\rVert_g^2 \,dx\quad\forall f\in C^{\infty}_c(X),
\end{align}
where $\lVert\cdot\rVert_g$ is the norme induced by $g$. 

We can extend $\Delta$ to a self-adjoint operator on  $L^2(X)$ which in abuse of notation we will again denote by $\Delta$. 
Note that Definition \ref{defi:har} is equivalent to the existence of radial solutions of the Laplace equation, $\Delta f=0$, on punctured balls, stated in the introduction, which was the initial starting point for the study of harmonic manifolds by Copson and Ruse \cite{copson_ruse_1940}.

From this point onwards, unless specifically mentioned otherwise, we will assume $(X,g)$ to be a non-compact simply connected harmonic manifold of dimension $n$. 
 
 The next statement is classical it can be found in \cite[Sec. 6.8] {willmore1996riemannian}. 
 \begin{prop}\label{prop:ana}
A simply connected non-compact harmonic manifold $(X,g)$ is Einstein and therefore the Riemannian metric $g$ and all functions derived from g are analytic. Especially the Laplace operator $\Delta$ is an analytical elliptic operator.
\end{prop}

\begin{defi}
Let $v\in SX$ then the Busemann function $b_v:X\to \R$  is defined by $b_v=\lim_{t\to \infty} d(x,c_v(t))-t.$ Furthermore, the level sets of the Busemann function are called horospheres and are denoted by 
$H_v^s:=b^{-1}_v(s)$ for $s\in \R$. 
\end{defi}
Using the inverse triangle equation it is easy to see that the Busemann functions are $1$-Lipschitz and therefore their differential exists almost everywhere. On harmonic manifolds, this statement can be drastically improved. 
 \begin{bem}
 From \cite{BusemannHarmonic} we have 
\begin{enumerate}
\item The Busemann functions exist and are analytic. 
\item The mean curvature of the horospheres is constant and given by $\Delta b_v:=h:=2\rho$. 
\end{enumerate}
Furthermore, the authors in \cite[Corollary 5.2]{PS15} showed that the top of the spectrum of $\Delta$ is given by $-\rho^2$.
 \end{bem}

 \begin{lemma}[\cite{Biswas2019}, Lemma 3.1]\label{lemma:Lapace}
Let $f$ be a $C^2$ function on $(X,g)$  and $u$ a $C^{\infty}$ function on $\R$. Then we have:
$$\Delta (u\circ f)=(u''\circ f)\lVert \operatorname{grad} f\rVert_g^2 +(u'\circ f)\Delta f.$$
where $\lVert\cdot\rVert_g^2=g(\cdot,\cdot)$.
\end{lemma}

 With Lemma \ref{lemma:Lapace} we can calculate the spherical and horospherical part of the Laplacian, by
choosing $f=d_x$ for some $x\in X$. We obtain with $\Delta d_x(r)=\frac{A^{\prime}(r)}{A(r)}\circ d_x(r)$ using spherical coordinates around $x$
\begin{align}\label{equ:radiallaplace}
\Delta(u\circ d_x)=u^{\prime\prime}\circ d_x +u^{\prime}\circ d_x\cdot \frac{A^{\prime}}{A}\circ d_x.
\end{align}
For the Busemann function  $f=b_v$ with $\Delta b_v=h=2\rho$ we obtain using horospherical coordinates 
\begin{align}\label{eq:horolaplace}
\Delta( u\circ b_v)=u^{\prime\prime}\circ b_v+h\cdot u^{\prime}\circ b_v.
\end{align}

\section{The Convulotion on Harmonic manifolds}
Fix $x_{0}\in X$.
Denote by $\mathcal{E}$ the space of smooth functions on $X$ and by $\mathcal{D}$ the space of smooth functions with compact support in $X$. And by $\mathcal{E}_{0}$ resp. $\mathcal{D}_{0}$ the function in $\mathcal{E}$ rep. $\mathcal{D}$ radial around $x_{0}$. Furthermore denote by  $\mathcal{D}'$ $(\mathcal{D}_{0}'$) resp. $\mathcal{E}'$$(\mathcal{E}_{0}')$  the space of (radial) distributions on $X$ resp. (radial) distributions with compact support. We equip $\mathcal{E}$ rep. $\mathcal{D}$ with the usual topologies (see \cite[Chapter II Section 2]{helgason2000groups}), their subspaces with the induced topology and the spaces of distributions with the weak* topology. Lastly, we denote by $\langle\cdot,\cdot\rangle_{X}$ the canonical pairing between functions and distributions as well as the $L^{2}$ inner product on $X$, since when both are defined they coincide. 
\begin{defi}
For $f,g\in C^{0}(X)$ one with compact support and $g=u\circ d(x_{0},\cdot)$ for some $u:\R^{+}\to \C$, i.e. $g$  radial around $x_{0}$, define for $x\in X$ the $x$-translation by 
$$\tau_{x}g=u\circ d(x,\cdot)$$
and the convolution by:
$$
f*g(y)=\int_{X}f(z)\tau_{z}g(z)\,dz.
$$
\end{defi}
\begin{bem}
Note that the convolution has a meaningful extension to the $L^{p}$ spaces for $p\geq 1$.
For this observe that the integral above is defined for almost every $x\in X$, if $g$ and $f$ are in $L^1(X)$ and $g=u\circ d_{x_{0}}$ radial, since:
\begin{align*}
\lVert f*g\rVert_1&\leq \int_X\int_X\lvert f(y)\rvert\lvert(\tau_xg)(y)\rvert \,dy\,dx\\
&=\int_X\lvert f(y)\rvert\left(\int_{0}^{\infty}\lvert u(r)\rvert A(r) \,dr\right) \,dy\\
&=\lVert f\rVert_1\lVert g\rVert_1< \infty. 
\end{align*}
Moreover, by replacing $f\in L^1(X)$ by $f\in L^{\infty}(X)$ in the calculation above:
\begin{align*}
\lVert f*g\rVert_{\infty}\leq \lVert f\rVert_{\infty}\cdot \lVert g\rVert_1.
\end{align*}
Hence, by the Rize-Thorin theorem, it follows that 

 for all $ p\in[1,\infty]$ and $f\in L^p(X),$
\begin{align}\label{eq:conalg}
\lVert f*g\rVert_p\leq \lVert f\rVert_{p}\cdot\lVert g\rVert_1.
\end{align}

\end{bem}

\begin{lemma}[\cite{PS15},\cite{Biswas2019}]\label{lemma:con1}
\begin{enumerate}
\item For $f,g$ radial functions on $X$ such that their convolution is defined we have
$$ f*g=g*f$$
and $f*g$ is radial again. 
\item For $f$ a function on $X$, $g,h$ radial functions on $X$ we have:
$$f*(g*h)=(f*g)*h$$
whenever both sides are defined. 
\item The convolution extends to radial measures.
\item $L^{1}(X,x_{0})$ turns into a commutative subalgebra of the space of radial measures in $X$ if one equips both spaces with the convolution. 
\end{enumerate}
\end{lemma}
\begin{defi}
Let $x\in X$ then the mean value operators $M_{x}:C^{0}_{x}(X):\to C^{0}_{\text{even}}(\R)$ is defined as 
$$M_x f (r):=\frac{1}{\vol(S(x,r))}\int_{S(x,r)}f (z)\,dz$$
and the radialisation operator $R_{x}:C^{0}(X)\to C^{0}_{x}(X)$ is defined as
$$R_x(f)(y):=M_xf(d(x,y)).$$
\end{defi}

The following properties of the radialistation operator are well known for Damek-Ricci spaces \cite{Damek1992} and except for (5) also for general harmonic manifolds. 

\begin{prop}\label{prop:radialoperator}
Let $x\in X$. The Radialisation operator $R_{x}:C^{k}(X)\to C^{k}(X)$ has the following properties:
\begin{enumerate}
\item $R_{x}^{2}=R_{x}$.
\item $f\geq 0\Rightarrow R_{x}f\geq0$.  
\item $\langle R_{x}f,g\rangle_{X}=\langle f, R_{x}g\rangle_{X}$ for all $f,g\in C^{k}_{c}(X)$. 
\item $\int_{X}R_{x}f(y)\,dy=\int_{X}f(y)\,dy$ for all $f\in C_{c}^{k}(X)$. 
\item $R_{x}(f*R_{x}g)=R_{x}f*R_{x}g$ forall $f,g\in C_{c}^{k}(X).$ 
\item there exist a constant $c>0$ such that  $\operatorname{supp}f\subset B(x_{0},r)\Rightarrow \operatorname{supp}R_{x}f\subset B(x_{0}, cr)$ for all $f\in C^{k}_{c}(X)$.
\item $R_{x}$ extends to a bounded operator from $C^{k}(B(x_{0},r))$ to $C^{k}(B(x_{0}, cr))$ with $c$ as in $(6)$.
\item $\Delta R_{x}(f)=R_{x}(\Delta f)$ for all $f\in C^{2}(X)$.
\end{enumerate}
\end{prop}
(1)-(4) and (8) are well known for harmonic manifolds, see for instance \cite{szabo1990}. (6) is obvious for $c$ large enough and (7) follows from (6).
It remains to prove $(5)$ and for this, we need the following lemma. 
\begin{lemma}\label{lemma:raial3}
Let $f\in C_{c}^{k}(X)$ and $g,h\in C_{c}^{k}(X,x)$ then 
$$ \langle f*g,h\rangle_{X}=\langle f,g*h\rangle_{X}.$$
\end{lemma}
\begin{proof}
\begin{align*}
 \langle f*g,h\rangle_{X}&=\int_{X}f*g(z)h(z)\,dz\\
 &=\int_{X}\int_{X}f(y)\tau_{z}g(y)\,dy h(z)\,dz\\
&= \int_{X}\int_{X}f(y)\tau_{z}g(y)h(z)\,dy\,dz\\
&=\int_{X}\int_{X}f(y)\tau_{z}g(y)h(z)\,dz\,dy\\
&=\int_{X}f(y)\int_{X}h(z)\tau_{z}g(y)\,dz\,dy\\&
=\int_{X}f(y)\int_{X}h(z)\tau_{y}g(z)\,dz\,dy\\
&=\int_{X}f(y)h*g(y)\,dy\\
&=\langle f,h*g\rangle_{X}\\
&=\langle f,g*h\rangle_{X}
\end{align*}
\end{proof}
\begin{proof}[Proof of Proposition \ref{prop:radialoperator} (5)]
Let $f\in C_{c}^{k}(X)$ and $g,h\in C_{c}^{k}(X,x)$ then using Lemma \ref{lemma:raial3} we have:
\begin{align*}
\langle (R_{x}f)*g,h\rangle_{X}&=\langle R_{x}f,h*g\rangle_{X}\\
&=\langle f,h*g\rangle_{X}\\
&=\langle f*g,h\rangle_{X}\\
&=\langle R_{x}(f*g),h\rangle_{X}.
\end{align*}
Setting  for $g_{0}\in C_{c}^{k}(X)$ $g=R_{x}(g_{0})$ in the above yields the claim. 
\end{proof}

\begin{bem}
Note that the fact that the Laplace operator commutes with the mean value operator is equivalent to $X$ being harmonic. See for instance \cite[Lemma 1.1]{szabo1990} where a proof of the statement above can be found. 
\end{bem}

\begin{bem}
Form Proposition \ref{prop:radialoperator} (8) and equation (\ref{equ:radiallaplace}) it follows that for   $$L_{A}:=\frac{d^2}{dr^2}+\frac{A^{\prime}(r)}{A(r)}\frac{d}{dr}$$ 
we have:
$$L_A M_x(f)(r)=M_x(\Delta f)(r).$$
\end{bem}

\begin{bem}\label{rem:istrebution}
By the (sequential) density of $\mathcal{D}$ in $\mathcal{E}'$ and the (sequential) density of $\mathcal{E}'$ in $\mathcal{D}'$ the convolution extends to distributions and the properties in Lemma \ref{lemma:con1} are retained. Furthermore we have for $T\in \mathcal{E}'$ and $f\in \mathcal{E}_{0}$ 
\begin{align}\label{eq:conradial1}
T*f=(x\mapsto\langle T,\tau_{x}f\rangle_{X})\in \mathcal{E}
\end{align}
and for $T\in\mathcal{E}_{0}'$ and $f\in\mathcal{E}$
$$f*T=(y\mapsto \langle T, \tau_{x_{0}}R_{y} f\rangle_{X}),$$
where $\langle\cdot,\cdot\rangle_{X}$ denotes the $L^{2}$ inner product on $X$. 
Lastly we have for $S,T\in\mathcal{E}_{0}'$ that $S*T\in \mathcal{E}_{0}$ and for $f\in \mathcal{E}_{0}$:
$$\langle S*T,f\rangle_{X}=\langle S,x\mapsto \langle T,\tau_{x}f \rangle_{X}\rangle_{X}.$$
For of proof of the above see \cite[Section 2]{PS15}.
\end{bem}

\subsection{Convolution Invariant Operators }
We keep the notation of the previous sections. Recall that the authors in \cite{Biswas2019} showed that $L^1(X,{x_{0}})$ forms a commutative Banach algebra under convolution.  
\begin{defi}
Let $x_{0} \in X$ and $1\leq p<\infty$. A densely defined operator $D:\operatorname{Dom}(D)\subset L^p(X)\to  L^p(X)$ is said to commute with translation if the radial part of the domain $\operatorname{Dom}(D)$ of $D$ is invariant under translation i.e. $\tau_{x_{0}} (\operatorname{Dom}(D)_{\text{rad}})\subset \operatorname{Dom}(D)_{\text{rad}}$. An operator is called translation invariant if $D$ is closed and commutes with translation. 
\end{defi}
 
 \begin{satz}
 Let $x_{0}\in X$ and $D:\operatorname{Dom}(D)\subset L^1(X)\to L^1(X)$ be a closed linear operator. Suppose furthermore that for $f\in \operatorname{Dom}(D)$ and $g\in L^{1}(X,x_{0})$ we have:
 \begin{enumerate}
 \item $f*g\in \operatorname{Dom}(D)$
 \item $(Df)*g=D(f*g),$
 \end{enumerate}
 then $\operatorname{Dom}(D)$ is a right $L^1(X,x_{0})$ Banach module under the  convolution of functions and the norm $\lVert\cdot\rVert_D:=\lVert D\cdot\rVert_{L^1}+\lVert \cdot\rVert_{L^1}$, moreover we have $\operatorname{Dom}(D)*L^1(X,x_{0})\subset \operatorname{Dom}(D)$.
 \end{satz}
 \begin{proof}
 First note that by (1) and since $\lVert f*g \rVert_{L^1}\leq \lVert f\rVert_{L^1}\cdot  \lVert g\rVert_{L^1}$ we have that $\operatorname{Dom}(D)*L^1(X,{x_{0}})\subset \operatorname{Dom}(D)$ i.e. $\operatorname{Dom}(D)$ is a right ideal of $L^1(X,{x_{0}})$. Now we need to check that $\operatorname{Dom}(D)$ is a right module of $L^1(X,x_{0})$. We know that $L^1(X,x_{0})$ is a commutative Banach algebra and has an approximation of the identity, for instance the heat kernel (see Remak \ref{rem:heat}). Therefore by Theorem 32.22 in \cite{Hewitt1979} we only need to check that the norm $\lVert \cdot \rVert_D$ is submultiplicative. 
 \begin{align*}
 \lVert f*g\rVert_D&=\lVert f*g\rVert_{L^1}+\lVert D(f*g)\rVert_{L^1}\\
 &\leq \lVert f\rVert_{L^1}\cdot \lVert g\rVert_{L^1}+\lVert D(f)*g\rVert_{L^1}\\
 &\leq \lVert f\rVert_{L^1}\cdot \lVert g\rVert_{L^1}+\lVert D(f)\rVert_{L^1}\cdot \lVert g\rVert_{L^1}\\
 &=\lVert g\rVert_{L^1}\cdot \lVert f\rVert_D.
 \end{align*}
 This give us that $ D(T)*L^1(X,x_{0})$  is a closed sub-algebra of $\operatorname{Dom}(D)$, hence $\operatorname{Dom}(D)*L^1(X,x_{0})= \operatorname{Dom}(D)$.
 \end{proof}
\section{Charatarisation of Eigenfunctions of The Laplacian} 

The characterisation of eigenfunctions via the average of spheres for symmetric and non-symmetric Damek Ricci spaces is well known, see for instance \cite[Chapter II Proposition 2.6]{helgason1994geometric} and \cite[Proposition 2.5.2]{Naik2021}.
The purpose of this chapter is to extend this characterisation to non-compact  harmonic manifolds. 
 \begin{satz}\label{thm:cheigf}
 Let $(X,g)$ be a non-compact simply connected harmonic manifold with mean curvature of the horospheres $2\rho$. 
Then for $f\in \mathcal{E}$, $f$ satisfies $\Delta f=-(\lambda^2+\rho^2)f$ for $\lambda\in \C$ if and only if $\Pi_t(f)(x)=f(x)\cdot \varphi_{\lambda,x}(t)$ for $t\geq 0$ and every $x\in X$.
Where for $x\in X$, $t\geq 0$ and $f:X\to \C$ measurable
 \begin{align*}
 \Pi_t(f)(x):=\frac{1}{\vol (S(x,t))}\int_{S(x,t)}f(y)\,dy
 \end{align*}
and $S(x,t)$ is the geodesic sphere of radius $t$ around $x$.
\end{satz}
Note that 
 the proof is essentially contained in the following series of lemmata. 
Lemma \ref{lemma:avarage}  is classic and can be found in \cite[Proposition (6.185)]{willmore1996riemannian}.
\begin{lemma}[\cite{willmore1996riemannian}]\label{lemma:avarage}
For $f\in \mathcal{E}$ we have:
\begin{align*} 
 \Pi_t(f)(x)=f(x)+\frac{1}{2n}\Delta f(x) t^2+ O(t^4)\quad\text{ for } t\to 0,
 \end{align*}
 where $n=\operatorname{dim} X$.
 \end{lemma}
 \begin{bem}\label{bem:Beven}
 Let $x\in X$ and $v\in S_xX$. The Jacobi tensor $A_v(r)$ along the geodesic $c_v:\R\to X$ with initial conditions $A_v(0)=0$ and $A'_v(r)=\operatorname{id}$ is given by 
\begin{align*}
D \exp_p(rv)(tw)=A_v(r)w(r),
\end{align*}
where $w\in T_x X$ and $w(r)\in T_{c_v(r)}X$ is the parallel transport of $w$ along $c_v$.
Then $A(r)=\operatorname{det}A_v(r)$ is the Jacobian of the map $v\to \exp(rv)=\exp\circ (v\to rv)$.
Hence, 
\begin{align}\label{eq:AB}
A(r)=r^{n-1}\operatorname{det}(D\exp_p)_{rv}\quad\forall v\in S_x X.
\end{align}
Observe that since $X$ is harmonic, $\mathbf{B}(rv):=\operatorname{det}(D\exp_p)_{rv}$ is independent of the choice of $v\in SX $ and therefore $\mathbf{B}(v)=\mathbf{B}(-v)$. Hence, $B(r):=\mathbf{B}(rv)$ can be seen as the restriction of an even function on $\R$. And therefore
$$A(r)=r^{n-1}B(r),$$ for some even function $B(r)$. 
 \end{bem}
 
 \begin{lemma}\label{lemma:limitvarphi}
 Let $\lambda\in\C$ and $\varphi_{\lambda,x}$ be the eigenfunction of the Laplacian which is radial around and normalised at $x$ with eigenvalue $-(\lambda^2+\rho^2)$. Then:
 \begin{align*}
 \lim_{r\to 0}\frac{\varphi_{\lambda,x}(r)-1}{r^2}=-\frac{(\lambda^2+\rho^2)}{2}.
 \end{align*}
 \end{lemma}
 \begin{proof}
By equation (\ref{equ:radiallaplace}) the radial part of the Laplacian $L_{A}$ is given by 
\begin{align*}
 L_{A}=\frac{d^2}{dr^2}+\frac{A^{\prime}(r)}{A(r)}\frac{d}{dr}.
\end{align*}
 By Remark \ref{bem:Beven} we have $A(r)=r^{n-1}B(r)$ for some even function $B(r)$ with $B(0)=1$ and $B^{\prime}(0)=0$. 
Therefore we obtain:
\begin{align}\label{eq:limitvarphi1}
-(\lambda^2+\rho^2)\varphi_{\lambda,x}(r)=\varphi^{\prime\prime}_{\lambda,x}(r)+\Bigl ( \frac{n-1}{r}+\frac{B^{\prime}(r)}{B(r)}\Bigr ) \varphi^{\prime}_{\lambda,x}(r).
\end{align}
Since $\varphi_{\lambda,x}(0)=1$ and $\varphi^{\prime}_{\lambda,x}(0)=0$ we obtain by taking limits in  equation (\ref{eq:limitvarphi1}) whereby we employ L'Hopital:
\begin{align*}
-(\lambda^2+\rho^2)=n\cdot \varphi^{\prime\prime}_{\lambda,x}(0).
\end{align*}
Hence, by using L'Hopital two times over we obtain:
\begin{align*}
 \lim_{r\to 0}\frac{\varphi_{\lambda,x}(r)-1}{r^2}= \lim_{r\to 0}\frac{\varphi^{\prime}_{\lambda,x}(r)}{2r}=\frac{1}{2}\varphi^{\prime\prime}_{\lambda,x}(0).
 \end{align*}
This concludes the proof.
  \end{proof}
  
\begin{prop}\label{prop:eigf}
If $f\in \mathcal{E}$ satisfies 
\begin{align*}
\lim_{t\to 0} \frac{\Pi_t(f)(x)-\varphi_{\lambda,x}(t)f(x)}{t^2}=0
\end{align*}
for all $x\in X$. Then $\Delta f=-(\lambda^2+\rho^2) f$. 
\end{prop}
\begin{proof}
By Lemma \ref{lemma:avarage} and Lemma \ref{lemma:limitvarphi} we have:
\begin{align*}
0&=\lim_{t\to 0} \frac{\Pi_t(f)(x)-\varphi_{\lambda,x}(t)f(x)}{t^2}\\
&=\lim_{t\to 0} \frac{(\Pi_t(f)(x)-f(x))-(\varphi_{\lambda,x}(t)f(x)-f(x))}{t^2}\\
&=\lim_{t\to 0} \frac{(\Pi_t(f)(x)-f(x))}{t^2} -\lim_{t\to 0}\frac{(\varphi_{\lambda,x}(t)f(x)-f(x))}{t^2}\\
&=\frac{\Delta f(x)+(\lambda^2+\rho^2)f(x)}{2n}.
\end{align*}
Hence, $f$ is a real analytic eigenfunction of $\Delta$ with eigenvalue $-(\lambda^2+\rho^2)$.

\end{proof}

\begin{proof}[Proof of Theorem \ref{thm:cheigf}]
The "if" part follows immediately from Proposition \ref{prop:eigf}.
The "only if part" follows easily from the fact that $\Pi_t$ commutes with the radial Laplacian, if one interprets $\Pi_t(x)=M_x(t)$. 
In detail: Let $L_A:=\frac{d^2}{dt^2}+\frac{A^{\prime}(t)}{A(t)}\frac{d}{dt}$ be the radial part of the Laplacian 
\begin{align*}
L_A \Pi_t(f)(x)=\Pi_t(\Delta f)(x)=-(\lambda^2+\rho^2) \Pi_t(f)(x)\quad\forall x\in X,t\geq 0
\end{align*}
and $A_0(f)(x)=f(x)$. Hence, by the uniqueness of the radial eigenfunctions we have
\begin{align*}
 \Pi_t(f)(x)=f(x)\cdot \varphi_{\lambda,x}(t)
 \end{align*}
 for every $t\geq0$ and $x\in X$.
\end{proof}
\begin{lemma}\label{A:self}
Let $f,g\in \mathcal{E}$ such that either of the functions has compact support then:
$$\int_X \Pi_t(f)(x)g(x)\,dx=\int_X f(x)\Pi_t(g)(x)\,dx.$$
\end{lemma}
\begin{proof}
Let $x\in X$, $d\theta_x$ the normalized mesure in $S_X$, $\phi^t$ the geodesic flow on $SX$, $\pi:SX\to X$ the foot point projection and $d\mu_l$ the Liouville  measure  on $SX$. Then we have since $d\mu_l$ is flow invariant and $d\theta_x$ is invariant under $v\mapsto -v$:
\begin{align*}
\int_X \Pi_t(f)(x)g(x)\,dx&=\int_X \int_{S_xX}f(exp(tv))\,d\theta_xg(x)\,dx\\
&=\int_{SX}f(\pi(\phi^t(v))g(\pi(v))\,d\mu_l\\
&=\int_{SX}f(\pi(\phi(v))g(\pi^{-t}(v))\,d\mu_l\\
&=\int_Xf(x)\int_{S_xX}g(\exp(-tv))\,d\theta_x\\
&=\int_Xf(x)\int_{S_xX}g(\exp(tv))\,d\theta_x\\
&=\int_X f(x)\Pi_t(g)(x)\,dx.
\end{align*}
\end{proof}
\begin{cor}
Let $f:X\to\C$ be a continuous function such that for $\delta>$ we have 
\begin{align*}
\Pi_t f(x)=\varphi_{\lambda,x}(t)f(x)\quad\forall t\in (0,\delta),\forall x\in X.
\end{align*}
Then $\Delta f=-(\lambda^2+\rho^2)f$. 
\end{cor}
\begin{proof}
Let $h\in \mathcal{E}$ be radial around $x_{0}$ such that the support of $h$ lies in $B(x_{0},r)$ for some $r\leq \delta$ and such that   $$\int_{B(x_{0},r)}\varphi_{\lambda,x_{0}}(z)h(z)\,dz=1.$$
Using Lemma \ref{A:self} we obtain:
\begin{align*}
f*h(x)&=\int_X f(y)\tau_xh(y)\,dy\\
&=\int_{B(x,r)}f(y)\tau_xh(y)\,dy\\
&=\int_0^r \Pi_t(f)(x)h(t)A(t)\,dt\\
&=f(x)\int_0^r \varphi_{\lambda,x}(t) h(t)A(t)\,dt\\
&=f(x).
\end{align*}
Now we note that the regularity in $x$ in the above is entirely in $\tau_x h$.
Hence, $f\in C^{\infty}(X)$ and Proposition \ref{prop:eigf} applies. 
\end{proof}

\section{The Abel Transform and Its Dual}

Peyerimhoff and Samion discussed the Abel transform and its dual for radial functions as well as its connection to the radial Fourier transform in \cite{PS15}.  Since we need them in the following we will give a brief introduction to them. 
\begin{defi}
Let $x_0\in X$ and $v\in S_{x_0}X$ then  $H_{v}^s=b_v^{-1}(s)$ denote the horospheres

$$j:C^{\infty}_{\text{even}}(\R)\to \mathcal{E}$$
$$(jf)(x)=e^{-\rho b_v(x)}f(b_v(x))$$
and $$a:C^{\infty}_{\text{even}}(\R)\to \mathcal{E}_{0}$$
by $$a(f)(y)=M_{x_0}( j(f))\circ d(x_0,y).$$
The dual with respect to the $L^2$-inner product of $\R$, $\langle\cdot,\cdot\rangle_{\R}$, and $X$, $\langle\cdot ,\cdot \rangle_{X} $, is called the Abel transform and is denoted by $\mathcal{A}:\mathcal{E}'_{0}\to C^{\infty}_{\text{even}}(\R)'$. This means that for every $g\in\mathcal{E}_{0}$ 
 and $f\in C^{\infty}_{\text{even}}(\R)$  we have
\begin{align*}
\int_\R \mathcal{A}(g)(s)f(s)\,ds=\int_X g(x)a(f)(x)\,dx.
\end{align*}
\end{defi}

 \begin{bem}
 Furthermore the authors in  \cite{PS15} showed  that:
 \begin{enumerate}
\item For $f\in\mathcal{D}_{0}$ 
 we have: 
\begin{align*}
\mathcal{A}(f)(s)&=e^{-\rho s}\int_{H^s_v}f(z)\,dz=e^{\rho s}\int_{H^0_v}f(\Psi_{v,s}(z))\,dz.
\end{align*}
and $\mathcal{A}(f)$ is smooth, has compact support and is even. 

\item 
 The Euclidean Fourier transform of the Abel transform is equal to the radial Fourier transform, given  for a function radial around $x_0$ with compact support by 
$$\hat{f}^{x_0}(\lambda)=\int_Xf(x)\varphi_{\lambda,x_0}(x)\,dx,$$
where $\varphi_{\lambda,x_0}$ is the radial eigenfunction of the Laplacian around $x_0$ with eigenvalue $-(\lambda^2+\rho^2)$ and $\varphi_{\lambda,x_0}(x_0)=1$. 
This means that 
\begin{align}\label{eq:abel-fourier}
\hat{f}^{x_0}(\lambda)=\mathcal{F}(\mathcal{A}(f))(\lambda)
\end{align}
where $\mathcal{F}(u)(\lambda)=\int_{\R}e^{i\lambda s}u(s)\,ds$ for $u:\R\to\R$ sufficiently regular is the Euclidian Fourier transform.
\item Applying $\mathcal{F}^{-1}$ to both sides in equation (\ref{eq:abel-fourier}) yields that the Abel transform and thereby its dual are independent of the choice of $v\in S_{x_0}X.$
\end{enumerate}
\end{bem}

\begin{satz}[\cite{PS15}, Theorem 3.8]
The dual Abel transform is a topological isomorphism between the spaces of smooth even functions on $\R$ and smooth radial functions around $x_0$ and the Abel transform is a topological isomorphism between the dual spaces. When all spaces are equipped with the canonical topology. 
\end{satz}
 Lastly from Proposition 3.11 in \cite{PS15} we have:
\begin{lemma}\label{lemma:Abelconvolution}
Let $S,T\in \mathcal{E}_0$ then we have:
$$\mathcal{A}(S*T)=\mathcal{A}(S)*\mathcal{A}(T),$$ 
where $*$ on the right-hand side is the convulsion on the real line. 
\end{lemma}

\section{Differential Operators}
\begin{lemma}
Let $D:\mathcal{E}\to\mathcal{E}$ be a differential operator, leaving the space of radial function invariant, such that $\tau_{x}D=D\tau_{x}$ for all $x\in X$ then:

$$D(f*g)=Df*g=f*Dg\text{ for all }f\in \mathcal{D}_{0}\text{ and }g\in \mathcal{E}_{0}$$

\end{lemma}
\begin{proof}

The assertion follows from the calculation:
\begin{align*}
D(f*g)(x)&=D\int_{X}f(y)\tau_{x} g(y)\,dy\\
&=\int_{X}f(y)D\tau_{x}g(y)\,dy\\
&=\int_{X}f(y)\tau_{x}Dg(y)\,dy\\
&=(f*Dg)(x)
\end{align*}
and the fact that the convolution is commutative. 
\end{proof}
\begin{satz}\label{thm:diffop}
Let $D:\mathcal{E}\to\mathcal{E}$ be a differential operator such that $D=D^{*}$ (on Dom(D)) with regards to the $L^{2}$ inner-product, $\tau_{x}D=D\tau_{x}$ for every $x\in X$ 
and $D$ has a radial fundamental solution. If for every $f\in \mathcal{D}$
\begin{align}\label{eq:support}
\operatorname{supp}(Df)\subset V\Rightarrow \operatorname{supp}(f)\subset V, \quad V\subset X.
\end{align}
Then $Du=f$ has a solution $u \in \mathcal{E}$ for every $f\in \mathcal{E}.$ 
\end{satz}
\begin{bem}
Note that the existence of a fundamental solution implies that $D$ is surjective on $\mathcal{D}$ since we can obtain a solution by convolution. 
\end{bem}

We are going to use the following classical theorems by Banach to show Theorem \ref{thm:diffop} for the proofs see for instance  \cite{treves2006topological}.
\begin{satz}
Let $Y$ be a Frechet space and $A:Y\to Y$ a continous linear map then $A$ is surjectiv if and only if:
$A^{*}:Y^{*}\to Y^{*}$ is injectiv and $A^{*}(Y^{*})$ is closed in $Y^{*}$ with respect to the weak$*$ topology.
\end{satz}
\begin{satz}
 $A^{*}(Y^{*})$ is closed if and only if for every $B\subset A^{*}$ bounded  $A^{*}(Y^{*})\cap B$ is closed in $B$.
\end{satz}

\begin{lemma}\label{lemma:dual injektiv}
$D^{*}:\mathcal{E}'\to \mathcal{E}'$ is injektiv. 
\end{lemma}
\begin{proof}
Let $F$ be a radial fundamental solution of $D$ i.e. $f\in \mathcal{D}_{0}'$ such that $DF(x)=\delta(x)$.
Then by approximating $T\in \mathcal{E}'$ by  functions in $\mathcal{D}$ 
we have for  $$D^{*}T*F=T*DF=T*\delta=T.$$ Yielding the claim. 
\end{proof}

\begin{lemma}\label{lemma:supportdual}
If (\ref{eq:support}) holds then for $T\in \mathcal{E}'$:
$$ \operatorname{supp}(D^{*}T)\subset V\Rightarrow \operatorname{supp}(T)\subset V,\quad V\subset X.$$
\end{lemma}
\begin{proof}
Let $\varphi_{\epsilon}\in \mathcal{D}_{0}$ be a function with support in  the closed ball of radius $\epsilon$ around $x_{0}$ denoted by  $B(x_{0},\epsilon)$
such that  $$\int_{X}\varphi_{\epsilon}(x)\,dx=1.$$
If $T\in \mathcal{E}'$ then by definition  $D^{*}T(f)=\langle T, Df\rangle_{X}$ for $f\in Dom (D)$. Suppose $\operatorname{supp}D^{*}T\subset V$ for some $V\subset X$ 
and cosider $T*\varphi_{\epsilon}(x)=\langle T, \tau_{x}\varphi_{\epsilon}\rangle$. Then $D(T*\varphi_{\epsilon})=D^{*}T*\varphi_{\epsilon}=T*D\varphi_{\epsilon}$ since $T*\varphi_{\epsilon}\in \mathcal{E}$. 
Furthermore we can assume $V=B(x_{0},R)$. Then since $T$ has support in $B(x_{0},R)$ and $\varphi_{\epsilon}$ has support in $B(x_{0} ,\epsilon)$ we have by the definition of the convolution
 $$\operatorname{supp}T*\varphi_{\epsilon} \subset B(x_{0},R+\epsilon).$$ Furthermore we have for $f\in \mathcal{D}$ that $(T*\varphi_{\epsilon})(f)=T(f*\varphi_{\epsilon})$. Now $\lim_{\epsilon \to 0}f*\varphi_{\epsilon}=f$ for all $f\in \mathcal{D}$. Hence $\operatorname{supp} T \subset B(x_{0},R)$. 
\end{proof}
\begin{lemma}\label{lemma:bounded}
Let $C\subset X$ be compact, then $\mathcal{E}'_{C}$  denotes the set of all elements in $\mathcal{E}'$ with support in $C$.
If $B'\subset \mathcal{E}'$ is bounded in the (Frechet sense) then $B'$ is contained in a set $\mathcal{E}'_{C}$ for some compact subset $C\subset X$. 
\end{lemma}
\begin{proof}
$B'$ is closed if and only if for every neighbourhood $U$ of the origin there is a $k>0$ such that $B'\subset k U$, since the set $\mathcal{E}'_{D}$, $D\subset X$ being a compact set, is a basis of the topology of $\mathcal{E}'$ the claim follows. 
\end{proof}
\begin{lemma}
$(D\mathcal{E}')\cap \mathcal{E}_{C}'$ is closed in $\mathcal{E}'$. 
\end{lemma}
\begin{proof}
Let $T_{j}=D^{*}S_{j}$, where $S_{j}\in \mathcal{E}'$ is a net in $(D^{*}\mathcal{E}')\cap \mathcal{E}_{C}^{*}$ that convergese (in the weak* sense) to $T\in\mathcal{E}'$. We have to show  then $T\in \mathcal{E}'_{C}$. Choose $r>0$ such that $C\subset B(x_{0},r)$. Then $\operatorname{supp}(T_{i})\subset B(x_{0},r)$ hence by Lemma \ref{lemma:supportdual} $\operatorname{supp} S_{j}\subset B(x_{0},r)$. If $J^{*}\in\mathcal{D}'$ is a radial fundamental solution of $D^{*}$ we have by the injectivity:
\begin{align}\label{eq:ST}
S_{j}=T_{j}*J^{*}\in \mathcal{E}'
\end{align}
which converges in the weak* topologie to $T*J^{*}$ in $\mathcal{E}'$. Hence $\operatorname{supp}(T*J^{*})\subset B(x_{0},r)$. Applying $D^{*}$ to (\ref{eq:ST}) yields:
$$ T_{j}=D^{*}(T_{j}*J^{*}).$$
Taking limits we obtain:
$$ T=D^{*}(T*J^{*}).$$ This yields the claim since differential operators are local. 
\end{proof}
\begin{proof}[Theorem \ref{thm:diffop}]
The Theorem follows now by applying the theorems of Banch to results from Lemma \ref{lemma:dual injektiv} and Lemma \ref{lemma:supportdual}.
\end{proof}

\subsection{Existance of Fundamental Solutions for Invariant Differential Operators with Constant Horospherical Part}

\begin{satz}\label{thm:fundamental}
Let $D:\mathcal{E}\to \mathcal{E}$ be a self-adjoint differential operator invariant under averages, i.e. $R_{x_0}D=DR_{x_0}$ for all $x_0\in X$ and with constant coefficients on horopheres, meaning that in horopherical coordinates $D$ is made up of a constant part  (not depending on the direction) and a part tangent to the horophere. Then $D$ has a radial fundamental solution. 
\end{satz}
\begin{lemma}\label{lemma1:fundamental}
Let $\delta_{x_0}$ be the delta distribution at $x_0\in X$ then 
$$\mathcal{A}(\delta_{x_0})=\delta_0,$$
where $\delta_{0}$ is the delta distribution at $0\in \R$. 
\end{lemma}
\begin{proof}
Let $f$ be a smooth even function on $\R$ then
$$\langle \mathcal{A}(\delta_{x_0}),f\rangle_{\R}=\langle \delta_{x_0},a(f)\rangle_{X}=af(x_0)=R_{x_0}(j(f))(x_0)=f(0).$$
\end{proof}
\begin{lemma}\label{lemma2:fundamental}
Let $D$ be as in Theorem \ref{thm:fundamental}
For $f\in \mathcal{E}_0'$ we have 
$$ \mathcal{A}(Df)(s)=\widetilde{D}(\mathcal{A}f)(s),$$ 
where $\widetilde{D}=\sum_{k=0}^{M} a_k\frac{\partial^k}{\partial s^k}$, $a_k\in \C$ for some $M\in \N$. 
\end{lemma}
\begin{proof}
Let $u\in C^{\infty}_{even}(\R)$ then $ju\in \mathcal{E}$ is constant on horopsheres. Consequently:
$$ Dju=j(\widetilde{D}u)$$ for $\widetilde{D}$ of the requiert form. 
For $f\in \mathcal{E}_{0}'$ we obtain using previously established duality arguments and the properties of $D$: \begin{align*}
\langle \mathcal{A}Df,u\rangle_{\R}&=\langle Df,au\rangle_{X}\\
&=\langle f, D au\rangle_X\\
&=\langle f,D R_{x_0}ju\rangle_X\\
&=\langle f,R_{x_0}Dju\rangle_X\\
&=\langle f,R_{x_0}j\widetilde{D}u\rangle_X\\
&=\langle f, a(\widetilde{D}u)\rangle_X\\
&=\langle \mathcal{A}f, \widetilde{D}u\rangle_{\R}\\
&=\langle \widetilde{D}\mathcal{A}f,u\rangle_{\R}.\\
\end{align*}
The final step utilizes the fact that $f$ and thus $\mathcal{A}f$ have compact support and partial integration.

This yields the claim since $u$ is arbitrary. 
\end{proof}

\begin{proof}[Proof Theorem \ref{thm:fundamental}]
By the Malgrange-Ehrenpreis theorem (see for instance \cite{mitrea2018distributions}) there exists a fundamental solution $F_{\widetilde{D}}$ for $\widetilde{D}$. We define $F_D:=\mathcal{A}^{-1}F_{\widetilde{D}}$. 
Now we have from Lemma \ref{lemma2:fundamental}
$$ \mathcal{A}(DF_D)=\widetilde{D}\mathcal{A}F_D=\delta.$$
Hence the claim follows by Lemma \ref{lemma1:fundamental}. Moreover, we have, as one can check by Lemma \ref{lemma:Abelconvolution} and the fact that $R_x$ is seld adjoint,: 
$$u(x)=f*F_D(x)$$ 
is a solution of $Du=f$ for all $f\in \mathcal{D}$. 
\end{proof}
\section{The Algebra of Radial Invariant Operators}
\begin{satz}\label{thm:invariantalgebra}
Let $(X,g)$ be a non-compact simply connected harmonic manifold, $\Delta$ the Laplacian and $L_{A}$ its radial part. 
If $L$ is a  differential operator such that for all $x\in X$: 

 $$R_{x}L=LR_{x}.$$

Then $L$ is a Polynomial of the Laplace operator. 
\end{satz}
The proof is congruent to the proof of the statement for Damek-Ricci spaces  \cite{Damek1992} when one notices that the main arguments in there do not rely on the specific expression of $A(r)$ but rather on the fact that $A(r)$ is a radial function and that the mean curvature of geodesic spheres behaves like the flat case as $r\to 0$. 
\begin{lemma}\label{lemma:radialpoly0}
Let $h:(-\epsilon,\epsilon)\to \C$ be a smooth function ($C^{n}$ function) then
\begin{enumerate}
\item $L_{A}(h(r)r^{n})=h_{0}r^{n}+h_{1}r^{n-1}+h_{2}r^{n-2}$ for $n\geq 2$, where $h_{i}$ are smooth ($C^{n-2}$) one $(-\epsilon,\epsilon)\setminus\{0\}$ and $h_{i}r^{n-i}$ at least $C^{n-1}$ in $0$, with  $h_{i}r^{n-i}\to 0$ as $r\to 0$. 
\item $L_{A}^{k}(h(r)r^{(2k+1)})\to 0$ as $r\to 0$ for every $k\in \N.$
\end{enumerate}
\end{lemma}
\begin{proof}
We have 
\begin{align*}
L_{A}(h(r)r^{n})&=\frac{d^{2}}{dr^{2}}(h(r)r^{n})+\frac{A'(r)}{A(r)}\frac{d}{dr}(h(r)r^{n})\\
&=\tilde{h}_{0}(r)r^{n}+\tilde{h}_{1}r^{n-1}+\tilde{h}_{2}r^{n-2}+ \frac{A'(r)}{A(r)}(h'(r)r^{n}+nh(r)r^{n-1}),
\end{align*}
where $\tilde{h}_{i}$ are smooth functions ($C^{n-2}$) on $(-\epsilon,\epsilon)$. Hence we only need to check that $\frac{A'(r)}{A(r)}(h'(r)r^{n}+nh(r)r^{n-1})$ and its derivatives do not have singularetys at $0$ but this followes from the fract that $\frac{A'(r)}{A(r)}\sim \frac{1}{r}$ for small $r$.  This yields the first part. 
For the second part of the assertion consider:
\begin{align*}
L_{A}^{k}(h(r)r^{2(k+1)})&=L_{A}^{k-1}(h_{1,0}(r)r^{2(k+1)}+h_{1,1}(r)r^{2k+1}+h_{1,2}(r)r^{2k})\\
&=L_{A}(h_{k-1,0}(r)r^{2(k+1)}+\cdots + h_{k-1,(2k+2)-3}r^{3})\\
&=h_{k,0}r^{2(k+1)}+\cdots+h_{k,(2k+2)-2}r^{2},
\end{align*}
where the $h_{i,j}$ are obtained by unsing the first part interativly. This yields the claim by the first part. 
\end{proof}

\begin{lemma}\label{lemma:radialpoly}
Let
 
  $f\in \mathcal{D}$ be radial around $x_{0}\in X$ with $f=u\circ d_{x_{0}}$ then:
\begin{enumerate}
\item $u^{(2j-1)}(0)=0$ for all $j\in \N$.
\item For every $j\in \N$ there is a polynomial $P_{j}$ such that:  $u^{(2j)}(0)=P_{j}(\Delta)f(x_{0})$. Furthermore the polynomial $P_{j}$ is independent of the choice of radial function. 
\end{enumerate}
\end{lemma}
\begin{proof}
The first assertion follows from the fact that $u$ is even. 
For the second one, we take   the Taylor expansion of $u$ around $0$ to the $2k+1$ th term  using the Peano form of the remainder:
\begin{align}
u(r)=\sum_{j=0}^{k}\frac{u^{(2j)}(0)}{(2j)!}r^{2j}+h(r)r^{2(k+1)},
\end{align}
where $h:(-\epsilon,\epsilon)\to \C$ is a smooth function and $\lim_{r\to 0}h(r)=0$. 
Furthermore by Lemma \ref{lemma:radialpoly0} $L_{A}^{k}(h(r)r^{2(k+1)})$ vanishes as $r$ goes to $0$. 
 Hence there is  a constant $C(k)\neq 0$ such that $$L_{A}^{k}u(0)=C(k)u(0)^{2k}+\text{lower terms polynomial in  $u^{2j}(0)$ $j<k$}.$$
Therfore  $$\Delta^{k}f(x_{0})=\sum_{j=1}^{k}a_{j}u^{(2j)}(0),$$where $a_{k}\neq 0$. Now notice that for $k=1$ the polynomial is independent of the function. This yields the claim by induction. 
\end{proof}

\begin{proof}[Theorem \ref{thm:invariantalgebra}]
Let $L$ be a differential operator of order $M$ satisfying the conditions of Theorem \ref{thm:invariantalgebra}. Let $f=u\circ d_{x_{0}}$ be smooth with compact support. Using Lemma \ref{lemma:radialpoly} we obtain using Taylor expansion with the Peano form of the remainder:
\begin{align}\label{eq:polylaplace}
u(r)=u(0)+\sum_{2k\leq M}P_{k}(\Delta)f(x_{0})r^{2k}+h(r)r^{M+1},
\end{align}
for suitable polynomial $P_{k}$ and $h:(-\epsilon,\epsilon)\to \C$ a smooth function vanishing at $0$. 
Applying $L$ to (\ref{eq:polylaplace}) and noticing that fot $2k\leq M$ $Lr^{2k}$ converges to some constant for $r\to 0$ and $L(h(r)r^{M+1})\to 0$ as $r\to 0$ since its degree exceeds the order of $L$ yields:  

$$Lf(x_{0})=P(\Delta)f(x_{0}),$$ for some polynomial $P$. Now choose $\varphi\in \mathcal{D}$ arbitrary then:
$$(L-P(\Delta))R_{x_{0}}(\varphi)(x_{0})=0.$$ 
Since both commute with $R_{x_{0}}$ and $\varphi$ is arbitrary we have:
$$ (L-P(\Delta))\varphi(x_{0})=0.$$ 

 Let  $\varphi$ be an abetrarry smooth function on $X$ and $x\in X$, set $g=\tau_{x_{0}}R_x\varphi$ then:
\begin{align*}
(L-P(\Delta))\varphi(x)&=R_x((L-P(\Delta))\varphi)(x)\\
&=(L-P(\Delta))R_x\varphi(x)\\
&=(L-P(\Delta))\tau_xg(x)\\
&=(L-P(\Delta))g(x_{0})=0.
\end{align*}
This Yields the claim. 
\end{proof}
\begin{cor}
Denote by $\mathcal{L}(X)$ the algebra of all differential operators satisfying the conditions of Theorem \ref{thm:invariantalgebra} then:
\begin{enumerate}
\item $\mathcal{L}(X)$ is commutative.
\item The eigenfunctions of all operators in $\mathcal{L}(X)$ coincide with the the eigenfunctions of $\Delta$, especially the charactarisation of Theorem \ref{thm:cheigf} applies.
\item Every operator in $\mathcal{L}(X)$ has a fundamental solution.
\item Every operator in $\mathcal{L}(X)$ is formally self adjoint with respect to the $L^{2}$ inner product on $X$. 
\item For $L\in \mathcal{L}(X)$ we have $\tau_xL=L\tau_x$ for all $x\in X$. 
\end{enumerate}
\end{cor}
\begin{proof}
(1), (2), (4) and (5) are direct consequences from Theorem \ref{thm:invariantalgebra}. (3) needs some explaining: Since every $D\in \mathcal{L}(X)$  is a polynomial in the Laplacian it has constant coefficients on horospheres and it fulfils (5) hence we can apply Theorem \ref{thm:fundamental} which yields the claim. 
\end{proof}

\begin{bem}
Note that $\Delta \in \mathcal{L}(X)$ is equivalent to  $X$ being harmonic. Hence the condition that $\Delta$ generates $\mathcal{L}(X)$ is at least as strong as harmonicity. 
\end{bem}
\subsection{The heat kernel}

The heat kernel $h_t$ is the fundamental solution of the heat equation 
\begin{align*}
u:X&\times\R_{\geq 0}\to\C\\
\frac{\partial}{\partial t} u&=\Delta u,\\
u(x,0)&=f(x)\in \mathcal{D},
\end{align*}
meaning that 
$h:\R^+\times X\times X\to \R$ and $h_t=h(t,\cdot,\cdot)$ we have 
\begin{align*}
 \frac{\partial h_t}{\partial t}=\Delta h_t\\
h_t(x,\cdot)\to \delta_x\text{ as } t\to 0.
 \end{align*}
Furthermore, we have for all $x\in X$:
$$\int_Xh(t,x,y)\,dy=1.$$
and $h_t(x,\cdot)\in L^p(X)$ for all $p\geq 1$. 
The heat kernel is intimately linked to the Laplacian since it is the convolution kernel of the heat-semi group $e^{\Delta t}$, $t>0$, generated by the Laplacian.  
\begin{bem}\label{rem:heat}
By \cite{szabo1990} $(X,g)$ is harmonic if and only if for all $x\in X$, the heat kernel $h_t(x,\cdot)=h_t^x(\cdot)$ is a radial function.
\end{bem}
\begin{satz}
Let $x_{0}\in X$ and  $h_t^{x_{0}}$ be the heat kernel at $x_{0}$. Then $\operatorname{span}(\{h_t^{x_{0}}\}_{t>0})$ is dense in $L^1(X,x_{0})$. 
\end{satz}

\begin{proof}
By dualety it is sufficiant to show that for  $f\in L^{\infty}(X,x_{0})$ $\langle f,h_t^{x_{0}}\rangle=0$ for all $t>0$  implies that $f=0$ in $L^{\infty}(X)$. 
 First, we note that $h_t^{x_{0}}\in L^1(X,x_{0})$  and that by \cite{davies1990heat}  for every $N\in \N$ and $\alpha>0$ there exist a $C>0$ such that:
 \begin{align*}
 \lvert \Delta^N h_t^{x_{0}}(x)\rvert \leq C e^{-\alpha d(x,x_{0})}\forall t>0. 
 \end{align*}
 Since the heat kernel satisfies the heat equation, this yields that all time derivatives are bounded by an $L^1(X)$ function. 
 From this and the heat equation we obtain, using the Leibniz integral rule, for every $N\in \N$:
 \begin{align*}
 \Delta^N(f*h_t^{x_{0}})(x_{0})&=(f*\Delta^Nh_t^{x_{0}})(x_{0})\\
 &=\langle f,\Delta^Nh_t^{x_{0}}\rangle\\
 &=\langle f, \frac{\partial^N}{\partial t^N} h_t^{x_{0}}\rangle\\
  &= \frac{\partial^N}{\partial t^N}\langle f, h_t^{x_{0}}\rangle\\
  &=0.
  \end{align*}
  Now let $D$ be any differential operator on $X$. Then since $f*h_t^{x_{0}}$ is radial we have that:
  $$ D(f*h_t^{x_{0}})(x_{0})=D_r(f*h_t^{x_{0}})(x_{0}),$$
  where $D_r$ is the radial part of $D$. But by Theorem \ref{thm:invariantalgebra} $D_r$ is a polynomial in the Laplacian, hence 
  \begin{align}\label{eq:thmheat1}
  D_r(f*h_t^{x_{0}})(x_{0})=0,
  \end{align}
   by the consideration above. 
  The heat kernel is an analytic function since harmonic manifolds are analytic (Proposition \ref{prop:ana}), therefore $f*h_t^{x_{0}}$ is analytic for every $t>0.$ But (\ref{eq:thmheat1}) implies that all derivatise of $f*h_t^{x_{0}}$ vansih in $x_{0}$. Hence by the identity theorem for power series $f*h_t^{x_{0}}=0$ for all $t>0$ but $h_t^{x_{0}}$ is an approximation of the identity. Hence $f=0$. 
\end{proof}

\begin{qu}
Is there a connection between the rank of a non-flat harmonic manifold as introduced in \cite{knieper2009new} and the size of the minimal generating system of the algebra $\mathcal{L}$? 
This would provide an answer to the question of whether all non-flat harmonic manifolds are of rank one. 
\end{qu}

\footnotesize

\bibliography{literature}
\bibliographystyle{alpha}

\end{document}